\documentclass[]{amsart}
\usepackage{amsmath}
\usepackage{amsthm}
\usepackage{amssymb}
\usepackage{amscd}
\usepackage{amsfonts}
\usepackage{amsbsy}
\usepackage{epsfig,afterpage}

\newtheorem {theorem} {Theorem}%[section]
\newtheorem {proposition} [theorem]{Proposition}

\newtheorem {remark} [theorem]{Remark}

\newtheorem {definition} [theorem]{Definition}

\newcommand{\R}{\mathbb{R}}

\def \e {\varepsilon}

\usepackage[usenames,dvipsnames]{color}

\begin{document}

\title[Flow Curvature Method applied to Canard Explosion]
{Flow Curvature Method\\ applied to Canard Explosion}

\author[J.M. Ginoux, J. Llibre]
{Jean-Marc Ginoux$^1$ and Jaume Llibre$^2$}

\address{$^1$ Laboratoire {\sc Protee}, I.U.T. de Toulon,
Universit\'{e} du Sud, BP 20132, F-83957 La Garde cedex, France}
\email{ginoux@univ-tln.fr}

\address{$^2$ Departament de Matem\`{a}tiques,
Universitat Aut\`{o}noma de Barcelona, 08193 Bellaterra, Barcelona,
Spain} \email{jllibre@mat.uab.cat}

\subjclass{}

\keywords{Geometric Singular Perturbation Method, Flow Curvature
Method, singularly perturbed dynamical systems, canard solutions}

\begin{abstract}
The aim of this work is to establish that the bifurcation parameter
value leading to a \textit{canard explosion} in dimension two
obtained by the so-called \textit{Geometric Singular Perturbation
Method} can be found according to the \textit{Flow Curvature
Method}. This result will be then exemplified with the classical Van
der Pol oscillator.
\end{abstract}

\maketitle

\section{Introduction}

The classical geometric theory of differential equations developed
originally by Andronov \cite{Andro1}, Tikhonov \cite{Tikh} and
Levinson \cite{Lev} stated that \textit{singularly perturbed
systems} possess \textit{invariant manifolds} on which trajectories
evolve slowly, and toward which nearby orbits contract exponentially
in time (either forward or backward) in the normal directions. These
manifolds have been called asymptotically stable (or unstable)
\textit{slow invariant manifolds}\footnote{In other articles the {\it slow manifold} is the approximation of order $O(\e)$ of the {\it slow invariant manifold}.}. Then, Fenichel \cite{Fen5,Fen6,Fen7,Fen8}
theory\footnote{The theory of invariant manifolds for an ordinary  differential equation is based on the work of Hirsch, \textit{et al.} \cite{Hirsch}} for the \textit{persistence of normally hyperbolic invariant manifolds} enabled to establish the \textit{local invariance} of \textit{slow invariant manifolds} that possess
both expanding and contracting directions and which were labeled \textit{slow invariant manifolds}.

\smallskip

During the last century, various methods have been developed to compute the \textit{slow invariant manifold} or, at least an asymptotic expansion in power of $\e$.

\smallskip

The seminal works of Wasow \cite{Wasow}, Cole \cite{Cole}, O'Malley \cite{Malley1,Malley2} and Fenichel \cite{Fen5,Fen6,Fen7,Fen8} to name but a few, gave rise to the so-called \textit{Geometric
Singular Perturbation Method}. According to this theory, existence as well as local invariance of the \textit{slow invariant manifold} of \textit{singularly perturbed systems} has been stated. Then, the
determination of the \textit{slow invariant manifold} equation turned into a regular perturbation problem in which one
generally expected the asymptotic validity of such expansion to breakdown \cite{Malley2}.

\smallskip

Recently a new approach of $n$-dimensional singularly perturbed dynamical systems of ordinary differential equations with two time scales, called \textit{Flow Curvature Method} has been developed \cite{Gin}. In dimension two and three, it consists in considering the \textit{trajectory curves} integral of such systems as \textit{plane} or \textit{space} curves. Based on the use of local metrics properties of \textit{curvature} (\textit{first curvature})
and \textit{torsion} (\textit{second curvature}) resulting from the \textit{Differential Geometry}, this method which does not require the use of asymptotic expansions, states that the location of the points where the local \textit{curvature} (resp. \textit{torsion}) of \textit{trajectory curves} of such systems, vanishes, directly provides an approximation of the \textit{slow invariant manifold} associated with two-dimensional (resp. three-dimensional) \textit{singularly perturbed systems} up to suitable order $O(\e^2)$ (resp. $O(\e^3)$). This method gives an implicit non intrinsic equation, because it depends on the euclidean metric.

\smallskip

Solutions of ``canard'' type have been discovered by a group of French mathematicians \cite{BenCallDinDin} in the beginning of the eighties while they were studying relaxation oscillations in the classical Van der Pol's equation (with a constant forcing term) \cite{VdP}. They observed, within a small range of the control parameter, a fast transition for the amplitude of the limit cycle varying suddenly from small amplitude to a large amplitude. Due to the fact that the shape of the limit cycle in the phase plane looks as a duck they called it ``canard cycle''. Hence, they named this new phenomenon ``canard explosion\footnote{According to Krupa and Szmolyan \cite[p. 312]{KrupSzmo} this terminology has been introduced in chemical and biological literature by Br$\varnothing$ns and Bar-Eli \cite[p. 8707]{Brons1991} to denote a sudden change of amplitude and period of oscillations under a very small range of control parameter.}'' and triggered a ``duck-hunting''.

\smallskip

Many methods have been developed to analyze ``canard'' solution such as nonstandard analysis \cite{BenCallDinDin, Diener}, matched asymptotic expansions \cite{Eckhaus}, or the blow-up technique \cite{DumoRouss, KrupSzmo,SzmoWechs} which extends the \textit{Geometric Singular Perturbation Method} \cite{Fen5,Fen6,Fen7,Fen8}.

\smallskip
Meanwhile, two other geometric approaches have been proposed. The first, elaborated by \cite{Brons1994} involves \textit{inflection curves}, while the second makes use of the \textit{curvature} of the trajectory curve, integral of any $n$-dimensional singularly perturbed dynamical system \cite{GiRo3,Gin}. This latter, entitled \textit{Flow Curvature Method} will be used in this work in order to compute the bifurcation parameter value leading to
a \textit{canard explosion}. Moreover, the total correspondence between the results obtained in this paper for two-dimensional singularly perturbed dynamical systems such as Van der Pol oscillator and those previously established by \cite{BenCallDinDin} will enable to highlight another link between the \textit{Flow Curvature Method} and the \textit{Geometric Singular Perturbation Method}.

\section{Singularly perturbed systems}

According to Tikhonov \cite{Tikh}, Takens \cite{Tak}, Jones \cite{Jones} and Kaper \cite{Kaper}
\textit{singularly perturbed systems} may be defined such as:

\begin{equation}
\label{eq1}
\begin{array}{*{20}c}
 {{\vec {x}}' = \vec {f}\left( {\vec {x},\vec {z},\varepsilon }
\right),\mbox{ }} \hfill \\
 {{\vec {z}}' = \varepsilon \vec {g}\left( {\vec {x},\vec {z},\varepsilon }
\right)}. \hfill \\
\end{array}
\end{equation}

where $\vec {x} \in \mathbb{R}^m$, $\vec {z} \in \mathbb{R}^p$,
$\varepsilon \in \mathbb{R}^ + $, and the prime denotes
differentiation with respect to the independent variable $t$. The
functions $\vec {f}$ and $\vec {g}$ are assumed to be $C^\infty$
functions\footnote{In certain applications these functions will be
supposed to be $C^r$, $r \geqslant 1$.} of $\vec {x}$, $\vec {z}$
and $\varepsilon$ in $U\times I$, where $U$ is an open subset of
$\mathbb{R}^m\times \mathbb{R}^p$ and $I$ is an open interval
containing $\varepsilon = 0$.

\smallskip

In the case when $0 < \varepsilon \ll 1$, i.e., is a small positive number, the variable $\vec {x}$ is called \textit{fast} variable, and $\vec {z}$ is called \textit{slow} variable. Using Landau's notation: $O\left( {\varepsilon^k} \right)$ represents a function $f$ of $x$ and $\varepsilon $ such that $f(u,\e)/\e^k$ is bounded for positive $\e$  going to zero, uniformly for $u$ in the given domain.

\newpage

It is used to consider that generally $\vec {x}$ evolves at an $O\left( 1 \right)$ rate; while $\vec {z}$ evolves at an $O\left( \varepsilon \right)$ \textit{slow} rate. Reformulating system (\ref{eq1}) in terms of the rescaled variable $\tau = \varepsilon t$, we obtain

\begin{equation}
\label{eq2}
\begin{aligned}
\varepsilon \dot {\vec {x}} & = \vec{f} \left( {\vec{x},\vec{z},
\varepsilon} \right), \\
\dot {\vec {z}} & = \vec {g}\left( {\vec{x}, \vec{z},\varepsilon }
\right).
\end{aligned}
\end{equation}

The dot represents the derivative with respect to the new independent variable $\tau $.

\smallskip

The independent variables $t$ and $\tau $ are referred to the \textit{fast} and \textit{slow} times, respectively, and (\ref{eq1}) and (\ref{eq2}) are called the \textit{fast} and \textit{slow} systems, respectively. These systems are equivalent whenever $\varepsilon \ne 0$, and they are labeled \textit{singular perturbation problems} when $0 < \varepsilon \ll 1$. The label ``singular'' stems in part from the discontinuous limiting behavior
in system (\ref{eq1}) as $\varepsilon \to 0$.

\smallskip

In such case system (\ref{eq2}) leads to a differential-algebraic system called \textit{reduced slow system} whose dimension decreases from $m + p = n$ to $p$. Then, the \textit{slow} variable $\vec {z} \in \mathbb{R}^p$ partially evolves in the submanifold $M_0$ called the \textit{critical manifold}\footnote{It corresponds to the approximation of the slow invariant manifold, with an error of $O(\e)$.} and defined by

\begin{equation}
\label{eq3} M_0 := \left\{ {\left( {\vec {x},\vec {z}} \right):\vec
{f}\left( {\vec {x},\vec {z},0} \right) = {\vec {0}}} \right\}.
\end{equation}

When $D_xf$ is invertible, thanks to implicit function theorem, $M_0 $ is given by the graph of a $C^\infty $ function $\vec {x} = \vec {F}_0 \left( \vec {z} \right)$ for $\vec {z} \in D$, where $D\subseteq \mathbb{R}^p$ is a compact, simply connected domain and the boundary of D is an $(p - 1)$--dimensional $C^\infty$ submanifold\footnote{The set D is overflowing invariant with respect to (\ref{eq2}) when $\varepsilon = 0$.}.

\smallskip

According to Fenichel theory \cite{Fen5, Fen6, Fen7, Fen8} if $0 < \varepsilon \ll 1$ is sufficiently small, then there exists a function $\vec {F}\left( {\vec {z},\varepsilon } \right)$ defined on D such that the manifold

\begin{equation}
\label{eq4} M_\varepsilon := \left\{ {\left( {\vec {x},\vec {z}}
\right):\vec {x} = \vec {F}\left( {\vec {z},\varepsilon } \right)}
\right\},
\end{equation}

is locally invariant under the flow of system (\ref{eq1}). Moreover, there exist perturbed local stable (or attracting) $M_a$ and unstable (or repelling) $M_r$ branches of the \textit{slow invariant manifold} $M_\varepsilon$. Thus, normal hyperbolicity of $M_\e$ is lost via a saddle-node bifurcation of the \textit{reduced slow system} (\ref{eq2}).

\begin{definition}
A ``canard'' is a solution of a singularly perturbed dynamical system following the \textit{attracting} branch $M_a$ of the \textit{slow invariant manifold}, passing near a bifurcation point located on the fold of the \textit{critical manifold}, and then following the \textit{repelling} branch $M_r$ of the \textit{slow invariant manifold} during a considerable amount of time.

\end{definition}

Geometrically a \textit{maximal canard} corresponds to the intersection of the attracting and repelling branches $M_a \cap M_r$ of the slow manifold in the vicinity of a non-hyperbolic point. Canards are a special class of solution of singularly perturbed dynamical systems for which normal hyperbolicity is lost.

\section{Geometric Singular Perturbation Method}

Earliest geometric approaches to \textit{singularly perturbed
dynamical systems} have been developed by Cole \cite{Cole}, O'Malley
\cite{Malley1,Malley2}, Fenichel \cite{Fen5,Fen6,Fen7,Fen8} for the
determination of the \textit{slow manifold} equation.

\smallskip

\textit{Geometric Singular Perturbation Method} is based on the
following assumptions and theorem stated by Nils Fenichel in the
middle of the seventies\footnote{For an introduction to Geometric
Singular Perturbation Method see \cite{Kaper}.}.

\subsection{Assumptions}
\begin{enumerate}

\item[$(H_1)$]
Functions $\vec {f}$ and $\vec {g}$ are $C^\infty $ functions in
$U\times I$, where $U$ is an open subset of ${\mathbb R}^m\times
{\mathbb R}^p$ and $I$ is an open interval containing $\varepsilon =
0$.

\item[$(H_2)$] There exists a set $M_0 $ that is contained in $\{\left( {\vec
{x},\vec {z}} \right):\vec {f}\left( {\vec {x},\vec {z},0}
\right)=0\}$ such that $M_0 $ is a compact manifold with boundary
and $M_0 $ is given by the graph of a $C^1$ function $\vec {x}=\vec
{F}_0 \left( {\vec {z}} \right)$ for $\vec {z}\in D$, where
$D\subseteq {\mathbb R}^p$ is a compact, simply connected domain and
the boundary of $D$ is an $(p-1)$--dimensional $C^\infty $
submanifold. Finally, the set $D$ is overflowing invariant with
respect to (\ref{eq2}) when $\varepsilon =0$.

\item[$(H_3)$] $M_0 $ is normally hyperbolic relative to the
\textit{reduced fast system} and in particular it is required for
all points $\vec {p}\in M_0 $, that there are $k$ (resp. $l)$
eigenvalues of $D_{\vec {x}} \vec {f}\left( {\vec {p},0} \right)$
with positive (resp. negative) real parts bounded away from zero,
where $k+l=m$.

\end{enumerate}

\begin{theorem}[Fenichel's Persistence Theorem]\label{th3.1}
Let system \eqref{eq1} satisfying the conditions $(H_{1}) -
(H_{3})$. If $\varepsilon >0$ is sufficiently small, then there
exists a function $\vec {F}\left( {\vec {z},\varepsilon }\right)$
defined on $D$ such that the manifold $M_\varepsilon =\{\left( {\vec
{x},\vec {z}}\right):\vec{x}=\vec {F}\left( {\vec {z},\varepsilon }
\right)\}$ is locally invariant under \eqref{eq1}. Moreover, $\vec
{F}\left( {\vec{z},\varepsilon } \right)$ is $C^r$, and $M_\varepsilon $ is $C^r$ $O\left(\varepsilon
\right)$ close to $M_0 $. In addition, there exist perturbed local
stable and unstable manifolds of $M_\varepsilon $. They are unions
of invariant families of stable and unstable fibers of dimensions
$l$ and $k$, respectively, and they are $C^r$ $O\left(
\varepsilon\right)$ close to their counterparts.
\end{theorem}

\begin{proof}
See \cite{Fen5}, \cite{Jones} and \cite{Kaper}.
\end{proof}

\subsection{Invariance}

Generally, Fenichel theory enables to turn the problem for explicitly finding functions $\vec {x}=\vec {F}\left( {\vec
{z},\varepsilon } \right)$ whose graphs are locally \textit{slow invariant manifolds} $M_\varepsilon $ of system (\ref{eq1}) into regular perturbation problem. Invariance of the manifold
$M_\varepsilon $ implies that $\vec {F}\left( {\vec {z},\varepsilon} \right)$ satisfies:

\begin{equation}
\label{eq5} \varepsilon D_{\vec {z}} \vec {F}\left( {\vec
{z},\varepsilon } \right)\vec {g}\left( {\vec {F}\left( {\vec
{z},\varepsilon } \right),\vec {z},\varepsilon } \right)=\vec
{f}\left( {\vec {F}\left( {\vec {z},\varepsilon } \right),\vec
{z},\varepsilon } \right).
\end{equation}

Then, plugging the perturbation expansion:

\[
\vec {F}\left( {\vec {z},\varepsilon } \right) = \sum_{i=0}^{N-1}\vec {F}_i \left(
{\vec {z}} \right)\e^i+ O\left( {\varepsilon^N} \right)
\]

into (\ref{eq5}) enables to solve order by order for $\vec {F}\left( {\vec {z},\varepsilon } \right)$.

\newpage

Taylor series expansion for $\vec {f}\left( {\vec {F}\left( {\vec
{z},\varepsilon } \right),\vec {z},\varepsilon } \right)$ up to
terms of order two in $\varepsilon $ leads at order $\varepsilon ^0$
to

\begin{equation}
\label{eq6} \vec {f}\left( {\vec {F}_0 \left( {\vec {z}}
\right),\vec{z},0} \right)=\vec {0}
\end{equation}

which defines $\vec {F}_0 \left( {\vec {z}} \right)$ due to the invertibility of $D_{\vec {x}} \vec {f}$ and the \textit{Implicit Function Theorem}.

\smallskip

At order $\varepsilon ^1$ we have:

\begin{equation}
\label{eq7} D_{\vec {z}} \vec {F}_0 \vec {g}\left( {\vec {F}_0,\vec
{z},0} \right) = D_{\vec {x}} \vec {f}\left( {\vec {F}_0,\vec {z},0}
\right)\vec {F}_1 + \frac{\partial \vec {f}}{\partial \varepsilon
}\left( {\vec {F}_0,\vec {z},0} \right),
\end{equation}

which yields $\vec {F}_1 \left( {\vec {z}} \right)$ and so forth.

\begin{equation}
\label{eq8} D_{\vec {x}} \vec {f}\left( {\vec {F}_0,\vec{z},0}
\right)\vec {F}_1 = D_{\vec {z}} \vec {F}_0 \vec {g}\left( {\vec
{F}_0, \vec {z},0} \right)-\dfrac{\partial \vec {f}}{\partial
\varepsilon}\left( {\vec {F}_0, \vec {z},0} \right).
\end{equation}

So, regular perturbation theory enables to build locally
\textit{slow invariant manifolds} $M_\varepsilon $. But for
high-dimensional \textit{singularly perturbed systems slow invariant
manifold} asymptotic equation determination leads to tedious calculations.

\begin{proof}
For application of this technique see \cite{Fen8}.
\end{proof}

\subsection{Slow invariant manifold and canards}

A manifold of canards is an invariant manifold, where first approximation is $M_0$.
For two-dimensional singularly perturbed dynamical systems with just one fast variable (\textit{x}) and one slow variable (\textit{y}), canards are non generic according to Krupa and Szmolyan \cite{KrupSzmo} and \textit{maximal canards} can only occur in such systems only for discrete values of a control parameter $\mu$. It means that in dimension two a one parameter family of singularly perturbed systems is needed to exhibit canard phenomenon. Because along a canard, the differential $D_xf$ is not always invertible, we can not write the manifold of canards as $x=F(y,\e)$. Thus, we will suppose that $D_yf$ is invertible and we will try to compute the canard as $y=F(x,\mu,\e)$. See \cite{Canalis} for a theory of this identification of formal series. We consider the following \textit{singularly perturbed dynamical system}:

\[
\begin{aligned}
\varepsilon \dot {x} & = f(x,y,\mu,\e), \\
\dot{y} & = g(x,y,\mu,\e), \\
\end{aligned}
\]

with $x,y\in \R$, i.e. $(m,p) = (1,1)$ and we suppose that due to the nature of the
problem perturbation expansions of the canard and of the canard value read:

\[
y = F(x, \varepsilon) = \sum_{i=0}^{N-1} F_i(x)\e^i + O\left( \varepsilon^N \right)
\quad \mbox{and} \quad
\mu(\e) = \sum_{i=0}^{N-1}\mu_i\e^i + O\left( \varepsilon^N \right).
\]

According to Eq. (\ref{eq5}) invariance of the manifold $M_{\varepsilon}$ reads:

\begin{equation}
\label{eq9} \left( \dfrac{\partial F}{\partial x}(x,\e) \right) f\left(
x,F(x,\e),\mu(\e),\e \right) = \varepsilon g\left(x, F(x,\e),\mu(\e),\e \right).
\end{equation}

\smallskip

To avoid technical complications in the computations below, we assume that, at order $O(\e^0)$, the critical manifold does not depend on the parameter $\mu$.

Indeed,

\[
\dfrac{\partial f}{\partial \mu}(x,F_0(x),\mu_0,0)=0
\]

Then, solving equation (\ref{eq9}) order by order provides at:

\smallskip

\textbf{Order $\varepsilon^0$}

\begin{equation}
\label{eq10}
\dfrac{\partial F_0}{\partial x}(x) f(x, F_0(x),\mu_0,0) = 0
\quad \Leftrightarrow \quad  f(x, F_0(x),\mu_0,0)=0.
\end{equation}

because the function $\dfrac{\partial F_0}{\partial x}(x)$ is almost everywhere non zero. Indeed, the function $F_0$ is given by the implicit function theorem. In what follows $f$, $g$, and their derivatives are evaluated at $(x, F_0(x), \mu_0, 0)$, and $F_0$, $F_1$ and $F_2$ are evaluated at $x$.\\

\smallskip

\textbf{Order $\varepsilon^1$}

\[
F_0'\left(\dfrac{\partial f}{\partial y}F_1+\dfrac{\partial f}{\partial \mu}\mu_1+\dfrac{\partial f}{\partial \e}\right)+F_1'f=g.
\]

Since according to what has been stated before, we have:

\begin{equation}
\label{eq11} F_1 = \dfrac{\dfrac{g}{F_0'}-\dfrac{\partial f}{\partial \e}}{\dfrac{\partial f}{\partial y}}
\end{equation}

A priori, this function is singular at the bifurcation point $x_0$ of the fast system, because $F_0'$ vanishes at this point. To avoid this singularity in function $F_1$, the relation $g(x_0,F(x_0),\mu_0,0)=0$ is needed. Whith an appropriate hypothesis on $\dfrac{\partial g}{\partial \mu}$, it  gives a value for $\mu_0$.\\

\smallskip

\textbf{Higher order}
The computation can be done with the same arguments. When condition of order $k$ are studied, we have to fix $F_k$, and to avoid singularity in $F_k$ we have to fix $\mu_{k-1}$. An example will be done in the next paragraph.

\subsection{Van der Pol's ``canards''}
Van der Pol system

\begin{equation}
\label{eq12}
\begin{aligned}
\varepsilon \dot {x} & = f(x,y) = x + y - \dfrac{x^3}{3}, \\
\dot{y} & = g(x,y) = \mu - x,
\end{aligned}
\end{equation}

satisfies Fenichel\index{Fenichel}'s assumptions $(H_1)-(H_3)$ except on the points $(x,y)=\pm(1,-\frac{2}{3})$.
The critical manifold is the cubic $y=x^3/3-x$. Thus, the problem is to find a function $y = F(x, \varepsilon)$
whose graph is locally the \textit{slow invariant manifold} $M_\varepsilon$ of the Van der Pol system. We write:

\[
F(x,\varepsilon) = F_0(x) + \varepsilon F_1 (x)+ \varepsilon^2 F_2
(x) + O\left( {\varepsilon ^3} \right)
\]

and

\hspace{3.5cm} $\mu = \mu_0 + \varepsilon \mu_1  + \varepsilon^2 \mu_2  + O\left({\varepsilon ^3} \right)$.\hfill

\newpage

The identification we have to perform is

\[
\sum_{i=0}\dfrac{\partial F_i}{\partial x}\e^i\left(\sum_{i=0}F_i\e^i-\left(\frac{x^3}{3}-x\right)\right)=
\sum_{i=0}\mu_i\e^{i+1}-\e x
\]

\vspace{0.1in}

Then, solving order by order provides at:

\vspace{0.1in}

\textbf{Order $\varepsilon^0$}

\[
F_0(x)= \dfrac{x^3}{3} - x.
\]

\textbf{Order $\varepsilon^1$}

\begin{equation}
\label{eq13} F_1(x) = - \dfrac{x - \mu_0}{x^2 - 1}
\end{equation}

\vspace{0.1in}

This function is singular at the fold point $x_0 = 1$ corresponding to the Hopf bifurcation point of the fast system\footnote{Due to the symmetry of the vector field: $(-x, -y, -\mu) \to (x, y, \mu)$ the same computation could have been done on the fold point $x_0 = -1$ in the vicinity of which a ``canard explosion'' also takes place.}. So, to avoid this singularity in function $F_1(x)$ we pose: $\mu_0 = 1$ and thus we have: $F_1(x) = \dfrac{-1}{1+x}$.

\vspace{0.1in}

\textbf{Order $\varepsilon^2$}

\begin{equation}
\label{eq14} F_2(x) = \dfrac{\mu_1 + \dfrac{(x - \mu_0)(x^2 + 1 -
2\mu_0x)}{(x^2 - 1)^3}}{x^2 - 1}
\end{equation}

\vspace{0.1in}

Taking into account that $\mu_0 = 1$ and, in order to avoid singularity in $F_2(x)$, we find that $\mu_1 = -\dfrac{1}{8}$ and so $F_2(x) = -\dfrac{x^2+4x+7}{8(1+x)^4}$.

\vspace{0.1in}

\textbf{Order $\varepsilon^3$}

\vspace{0.1in}

Using the same process, a tedious computation (or, better a computation with the help of a computer) leads to $\mu_2 = -\dfrac{3}{32}$, $\mu_3=-\dfrac{173}{1024}$. Thus, the bifurcation parameter value leading to canard solutions reads:

\[
\mu = 1 -\dfrac{1}{8} \varepsilon  -\dfrac{3}{32} \varepsilon^2  +
O\left( {\varepsilon ^3} \right)
\]

\vspace{0.1in}

Then, for $\varepsilon = 0.01$ one finds again the value obtained numerically by Beno\^{i}t \textit{et al.} \cite[p.99]{BenCallDinDin} and
Diener \cite{Diener}:

\[
\mu = 0.99874
\]

\newpage

The phenomenon of ``canard explosion'' of Van der Pol system (\ref{eq12}) with $\varepsilon = 0.01$ is exemplified on Fig. 1. where the periodic solution has been plotted in red, the critical manifold in black and the positive fixed point in green. Double arrows indicate the \textit{fast} motion while simple arrows indicate the \textit{slow} motion. Exponentially small variations of the parameter value $\mu = 0.99874$ enable to exhibit the transition from relaxation oscillation (a) to small amplitude limit cycles (b) via canard cycles (c). Then, at the parameter value $\mu = 1$ corresponding to the Hopf bifurcation, the canard disappears (d).

\begin{figure}[htbp]
  \begin{center}
    \begin{tabular}{ccc}
      \includegraphics[width=6cm,height=6cm]{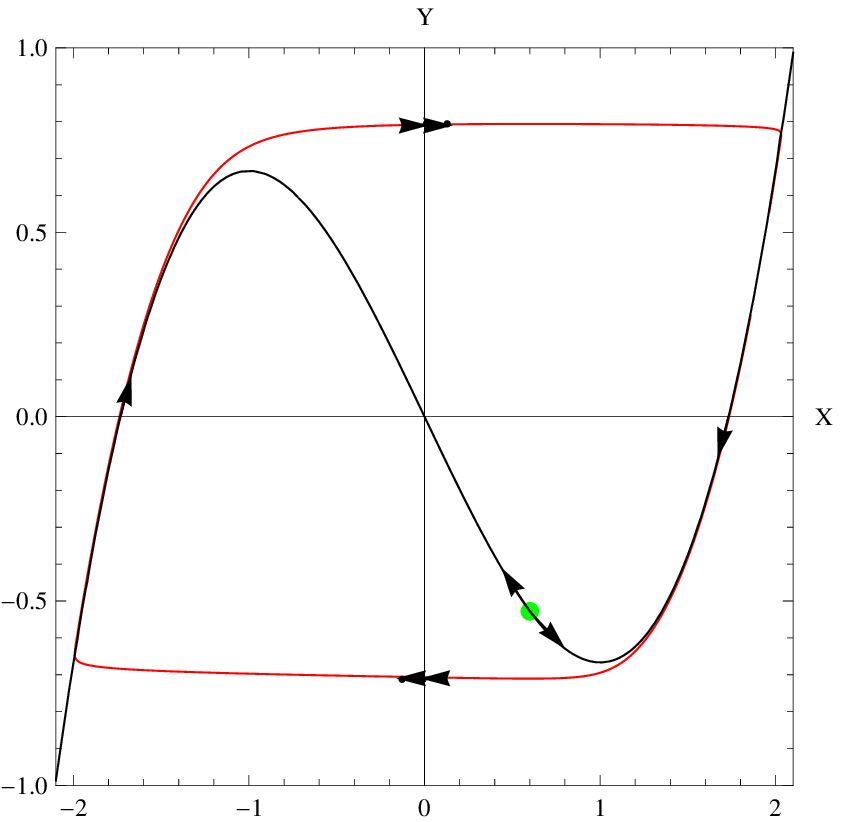} & ~~~~~~ &
      \includegraphics[width=6cm,height=6cm]{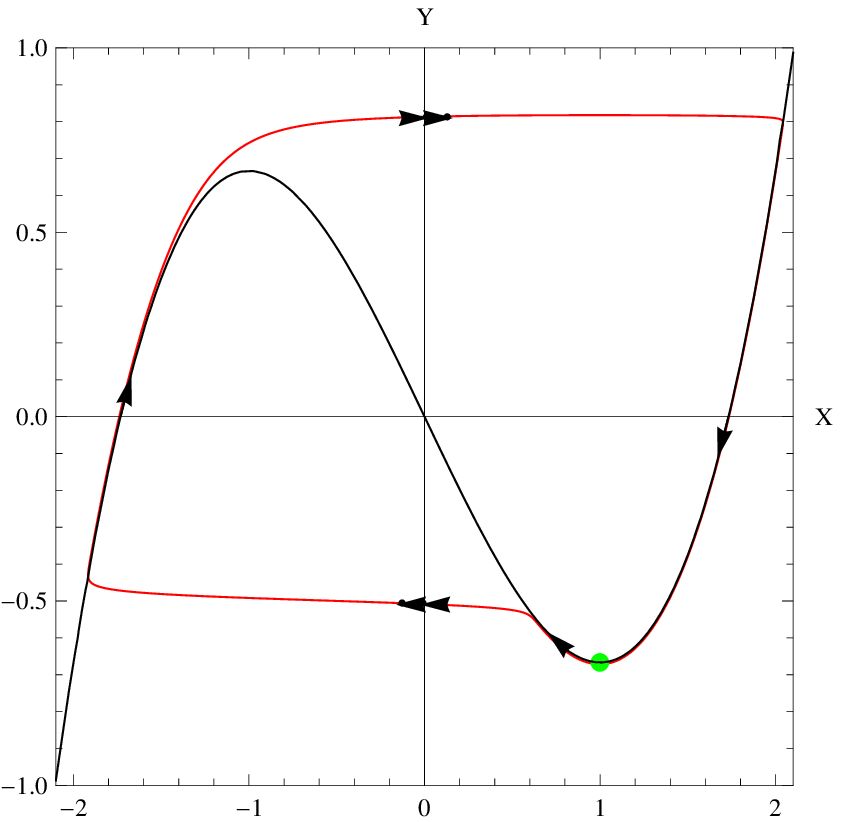} \\
      (a) $\mu < 0.99874$ & & (b) $\mu = 0.99874045$ \\[0.2cm]
      \includegraphics[width=6cm,height=6cm]{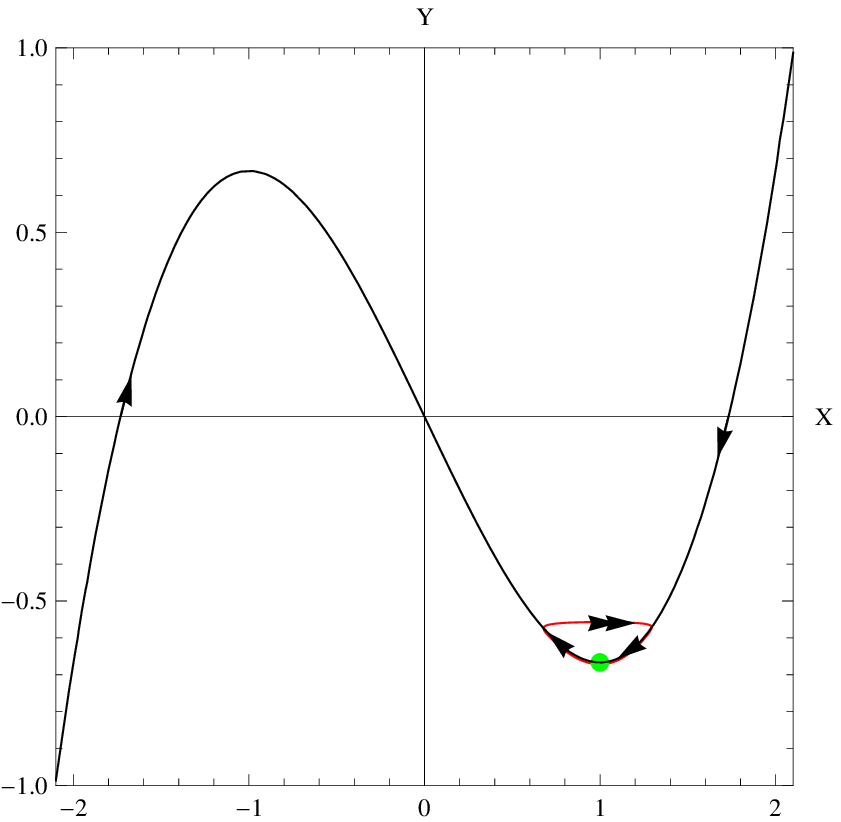} & ~~~ &
      \includegraphics[width=6cm,height=6cm]{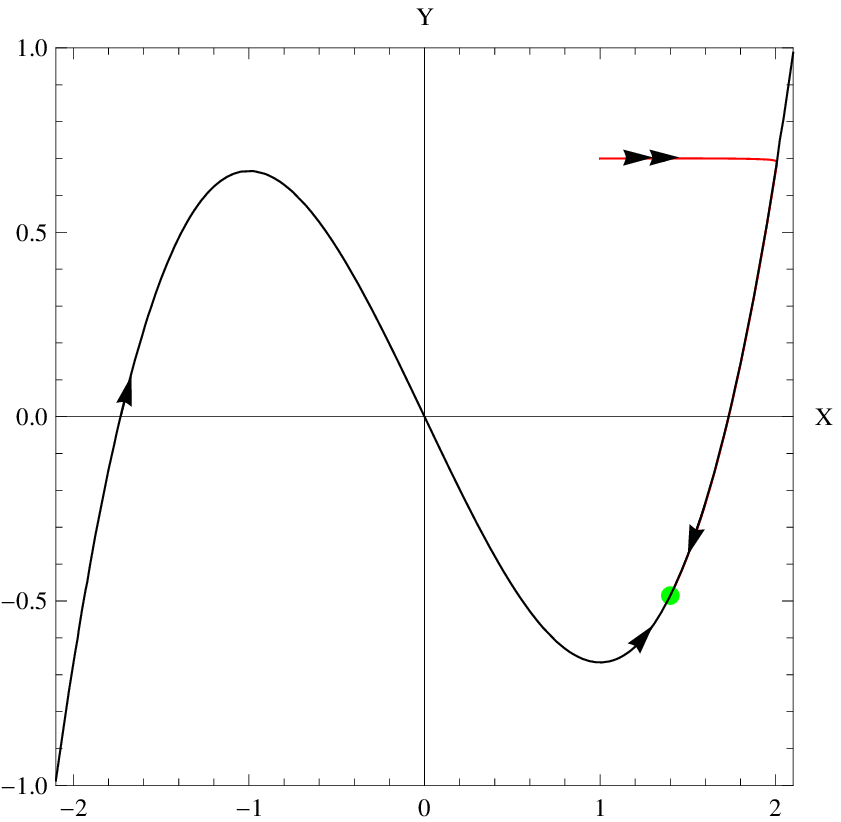} \\
      (c) $\mu = 0.998740451$ & & (d) $\mu > 1$ \\[-0.2cm]
    \end{tabular}
    \caption{Transition from relaxation oscillation to canard explosion.}
    \label{fig1}
  \end{center}
  \vspace{-0.5cm}
\end{figure}

\newpage

\section{Flow Curvature Method}

Recently, a new approach called \textit{Flow Curvature Method} and
based on the use of \textit{Differential Geometry} properties of
\textit{curvatures} has been developed, see \cite{GiRo3,Gin}.
According to this method, the highest \textit{curvature of the
flow}, i.e. the $(n - 1)^{th}$ \textit{curvature} of
\textit{trajectory curve} integral of $n$--dimensional dynamical
system defines a \textit{manifold} associated with this system and
called \textit{flow curvature manifold}. In the case of
$n$--dimensional singularly perturbed dynamical system (\ref{eq1}) for which with $\vec {x} \in \mathbb{R}^1$, $\vec {z} \in \mathbb{R}^{n-1}$, i.e. $(m,p)=(1,n-1)$ we have the following result.

\begin{proposition}
\label{prop3} The location of the points where the $(n - 1)^{th}$
\textit{curvature of the flow}, i.e. the \textit{curvature of the
trajectory curve} $\vec {X}$, integral of any $n$--dimensional
singularly perturbed dynamical system vanishes, represents its $(n -
1)$--dimensional \textit{slow manifold} $M_{\varepsilon}$ the
equation of which reads

\begin{equation}
\label{eq15} \phi ( {\vec {X}, \varepsilon} ) = \dot { \vec
{X} } \cdot ( {\ddot {\vec {X}} \wedge \dddot{\vec {X}}\wedge \ldots
\wedge \mathop {\vec {X}}\limits^{\left( n \right)} } ) = \det( {
\dot {\vec {X}},\ddot {\vec {X}}, \dddot{\vec {X}},\ldots ,\mathop
{\vec {X}}\limits^{\left( n \right)} } ) = 0
\end{equation}

where $\mathop {\vec {X}}\limits^{\left( n \right)}$ represents the
time derivatives up to order $n$ of $\vec {X}  = (\vec{x},
\vec{z})^t$.
\end{proposition}

\begin{proof}
For proof of this proposition see \cite[p. 185 and next]{Gin} and below.
\end{proof}

\begin{remark}
First, let's notice that with the Flow Curvature Method the slow manifold is defined by
an implicit equation. Secondly, in the most general case of $n$--dimensional singularly perturbed dynamical system (\ref{eq1}) for which $\vec {x} \in \mathbb{R}^m$, $\vec {z} \in \mathbb{R}^p$ the Proposition \ref{prop3} still holds. In dimension three, the example of a Neuronal Bursting Model (NBM) for which $(m,p)=(2,1)$ has already been studied by Ginoux \textit{et al.} \cite{GiRo2}. In this particular case, one of the hypotheses of the Tihonov's theorem is not checked since the fast dynamics of the singular approximation has a periodic solution. Nevertheless, it has been established by Ginoux \textit{et al.} \cite{GiRo2} that the slow manifold can all the same be obtained while using the Flow Curvature Method. According to this method, the slow invariant manifold of a three-dimensional singularly perturbed dynamical system for which $(m,p)=(1,2)$ is given by the $2^{nd}$ curvature of the flow, i.e. the torsion. In the case of a Neuronal Bursting Model for which $(m,p)=(2,1)$ it has been stated by Ginoux \textit{et al.} \cite{GiRo2} that the slow manifold is then given by the $1^{st}$ curvature of the flow, i.e. the curvature. In such a case, the flow curvature manifold is defined by the location of the points where the three-dimensional pseudovector $\dot{\vec{X}} \wedge \ddot{\vec{X}}$ vanishes. This condition leads to a nonlinear system of three equations two of which being linearly independent. These two equations define a curve corresponding to the slow invariant manifold\footnote{See also Gilmore \textit{et al.} \cite{Gilmore}}. Thus, one can deduce that for a three-dimensional singularly perturbed dynamical system for which $(m,p)$ the slow manifold is given by the $p^{th}$ curvature of the flow.
\end{remark}

\subsection{Invariance}

According to Schlomiuk \cite{Schlomiuk} and Llibre \textit{et al.}
\cite{LLibreMedrado} the concept of \textit{invariant manifold} has
been originally introduced by Gaston Darboux \cite[p. 71]{Darboux}
in a memoir entitled: \textit{Sur les \'{e}quations diff\'{e}rentielles
alg\'{e}briques du premier ordre et du premier degr\'{e}} and can be stated
as follows.

\begin{proposition}\label{prop4}
The \textit{manifold} defined by $\phi ( \vec {X}, \varepsilon) = 0$
where $\phi $ is a $C^1$ in an open set U, is \textit{invariant}
with respect to the flow of \eqref{eq1} if there exists a $C^1$
function denoted by $\kappa ( \vec {X}, \varepsilon)$ and called
cofactor which satisfies

\begin{equation}
\label{eq16}
L_{\overrightarrow V } \phi ( \vec {X}, \varepsilon) = \kappa( \vec
{X}, \varepsilon) \phi ( \vec {X}, \varepsilon),
\end{equation}

for all $\vec {X} \in U$, and with the Lie derivative operator
defined as

\[
L_{\overrightarrow V } \phi = \overrightarrow V \cdot
\overrightarrow \nabla \phi = \sum\limits_{i = 1}^n {\frac{\partial
\phi }{\partial x_i }\dot {x}_i } = \frac{d\phi }{dt}.
\]

\end{proposition}

\begin{proof}
According to \textit{Fenichel's Persistence Theorem} (see Th.
\ref{th3.1}) the \textit{slow invariant manifold} $M_{\varepsilon}$
may be written as an \textit{explicit} function $\vec{x} = \vec
{F}\left( {\vec {z},\varepsilon } \right)$, the invariance of which
implies that $\vec {F}\left( {\vec {z},\varepsilon } \right)$
satisfies

\begin{equation}
\label{eq17}
\varepsilon D_{\vec {z}} \vec {F}\left( {\vec {z},\varepsilon } \right)\vec
{g}\left( {\vec {F}\left( {\vec {z},\varepsilon } \right),\vec
{z},\varepsilon } \right)=\vec {f}\left( {\vec {F}\left( {\vec
{z},\varepsilon } \right),\vec {z},\varepsilon } \right)
\end{equation}

We write the \textit{slow manifold} $M_{\varepsilon}$ as an
\textit{implicit function} by posing

\begin{equation}
\label{eq18}
\phi(\vec {x},\vec {z},\varepsilon ) = \vec{x} - \vec {F}\left( \vec
{z},\varepsilon  \right) = \phi(\vec {X},\varepsilon ).
\end{equation}

According to \textit{Darboux invariance theorem} $M_{\varepsilon}$
is invariant if its Lie derivative reads

\begin{equation}
\label{eq19}
L_{\vec V } \phi (\vec {X},\varepsilon ) = \kappa(\vec
{X},\varepsilon ) \phi (\vec {X},\varepsilon ).
\end{equation}

Plugging Eq. (\ref{eq18}) into the Lie derivative (\ref{eq19}) leads
to

\[
L_{\vec V } \phi (\vec {X},\varepsilon ) = \dot{\vec{x}} - D_{\vec
{z}}\vec {F} ( \vec {z},\varepsilon ) \dot{\vec{z}} = \kappa(\vec
{X},\varepsilon ) \phi (\vec {X},\varepsilon ),
\]

which may be written according to Eq. (\ref{eq2}) as

\[
L_{\vec V } \phi (\vec {X},\varepsilon ) = \frac{1}{\varepsilon}(
\vec{f}(\vec {X},\varepsilon ) - \varepsilon D_{\vec {z}}\vec {F} (
{\vec {z},\varepsilon }) \vec{g}(\vec {X},\varepsilon ) ) =
\kappa(\vec {X},\varepsilon ) \phi (\vec {X},\varepsilon).
\]

Evaluating this Lie derivative in the location of the points where
$\phi(\vec {X},\varepsilon ) = 0$, i.e. $\vec{x} = \vec {F}( \vec
{z},\varepsilon )$ leads to

\[
L_{\vec V } \phi ( \vec {F}\left( \vec {z},\varepsilon \right),\vec
{z},\varepsilon ) = \frac{1}{\varepsilon}\left( \vec{f}(\vec
{F}\left( \vec {z},\varepsilon  \right),\vec{z},\varepsilon) -
\varepsilon D_{\vec {z}}\vec {F}\left( {\vec {z},\varepsilon }
\right) \vec{g}(\vec {F}\left( \vec {z},\varepsilon
\right),\vec{z},\varepsilon) \right) = 0,
\]

which is exactly identical to Eq. (\ref{eq18}) used by
Fenichel\index{Fenichel}.
\end{proof}

\begin{remark}
This last equation for the invariance of the manifold
$M_\varepsilon$ may be written in a simpler way which implies that
$\phi ( {\vec {x}, \vec {z},\varepsilon} )$ satisfies

\begin{equation}
\label{eq20} \dfrac{d}{dt} \left[ \phi ( {\vec {x}, \vec
{z},\varepsilon} ) \right] = 0,
\end{equation}

on the solutions of the differential system.
\end{remark}

\subsection{Slow invariant manifold}

We consider again the following two-dimensional \textit{singularly
perturbed dynamical system}

\[
\begin{aligned}
\varepsilon \dot {x} & = f(x,y), \\
\dot{y} & = g(x,y),  \\
\end{aligned}
\]

and we suppose that due to the nature of the problem perturbation
expansion reads

\[
y = F(x, \varepsilon) = F_0(x) + \varepsilon F_1(x) + \varepsilon^2
F_2(x) + O\left( \varepsilon^3 \right).
\]

According to the \textit{Flow Curvature Method} each function
$\vec{F_i}(\vec{z})$ of this perturbation expansion may be found
again starting from the \textit{slow manifold implicit equation}
(\ref{eq16}) as stated in the next result.

\begin{proposition}
The functions $F_i(x)$ of the \textit{slow invariant manifold}
associated with a two-dimensional singularly perturbed dynamical
system are given by the following expressions

\begin{equation}
\label{eq21} \begin{array}{*{20}c} F'_n (x) =\mathop
{\lim}\limits_{\varepsilon \to 0} \left[ {\dfrac{1}{n!}
\dfrac{\partial ^na_{10}}{\partial \varepsilon^n}} \right] \mbox{ with } n\ge 0, \vspace{4pt} \\
F_n (x) =\mathop {\lim}\limits_{\varepsilon \to 0} \left[
{\dfrac{1}{n!} \dfrac{\partial ^{n-1}a_{01}}{\partial
\varepsilon^{n-1}}} \right] \mbox { with } n\ge 1,
\end{array}
\end{equation}

where

\[
a_{10} =-\left. {\dfrac{\dfrac{\partial \phi_i}{\partial
x}}{\dfrac{\partial \phi_i}{\partial y}}} \right|_{y = F\left(
{x,\varepsilon }\right)} \quad \mbox{and} \qquad a_{01} =-\left.
{\dfrac{\dfrac{\partial \phi_i}{\partial
\varepsilon}}{\dfrac{\partial \phi_i}{\partial y}}} \right|_{y =
F\left( {x,\varepsilon }\right)},
\]

and

\[
\phi_i(x,y,\varepsilon) =
\frac{d^{i-1}}{dt^{i-1}}[\phi_1(x,y,\varepsilon)] \quad \mbox{with}
\quad \mbox{i = 1, 2, ..., n},
\]

where $\phi_i(x,y,\varepsilon)$ corresponds to the $i^{th}$ order
approximation in $\varepsilon$.

\end{proposition}

\begin{proof}

We have that

\[
\phi ( {\vec {X}, \varepsilon} )=\left\| {\dot {\vec {X}}\wedge
\ddot {\vec {X}}} \right\| = \det( { \dot {\vec {X}},\ddot {\vec
{X}} } ) = 0 \quad \Leftrightarrow \quad \phi \left(
{x,y,\varepsilon } \right)=0.
\]

Since, for a two-dimensional singularly perturbed dynamical systems
this \textit{slow manifold equation} is defined by the
\textit{second order tensor of curvature}, i.e. by a determinant
involving the first and second time derivatives of the vector field
$\vec {X}$, it corresponds to the first order approximation in
$\varepsilon$ of the \textit{slow manifold} obtained with the
\textit{Geometric Singular Perturbation Method}. So, we denote it by

\[
\phi_1 ( {\vec {X}, \varepsilon} )=\left\| {\dot {\vec {X}}\wedge
\ddot {\vec {X}}} \right\| = \det( { \dot {\vec {X}},\ddot {\vec
{X}} } ) = 0.
\]

A \textit{third order tensor of curvature} can be easily given by
the time derivative of $\phi_1 ( {\vec {X}, \varepsilon} )$. We
denote it by

\[
\phi_2 ( {\vec {X}, \varepsilon} )= \dot \phi_1 ( {\vec {X},
\varepsilon} ) = \det( { \dot {\vec {X}},\dddot {\vec {X}}}) = 0.
\]

Thus, $\phi_2 ( {\vec {X}, \varepsilon} )$ corresponds to the second
order approximation in $\varepsilon$. Using the same process, we
consider the \textit{slow manifold} $\phi_i( {\vec {X}, \varepsilon}
)$ which corresponds to the $i^{th}$ order approximation in
$\varepsilon$.

Writing the \textit{total differential} of the \textit{slow
manifold} we obtain

\begin{equation}
\label{eq22} d\phi_i \left( {x,y,\varepsilon }
\right)=\dfrac{\partial \phi_i }{\partial x}dx+\dfrac{\partial
\phi_i }{\partial y}dy + \dfrac{\partial \phi_i}{\partial
\varepsilon} d\varepsilon =0.
\end{equation}

Replacing in Eq. (\ref{eq23}) $dy$ by its \textit{total
differential} $dy = \dfrac{\partial F}{\partial x} dx +
\dfrac{\partial F}{\partial \varepsilon} d\varepsilon$ yields

\begin{equation}
\label{eq23} d\phi_i(x,y,\varepsilon) = \left( \dfrac{\partial
\phi_i}{\partial x} + \dfrac{\partial \phi_i}{\partial
y}\dfrac{\partial F}{\partial x}  \right)dx + \left( \dfrac{\partial
\phi_i}{\partial \varepsilon} + \dfrac{\partial \phi_i}{\partial
y}\dfrac{\partial F}{\partial \varepsilon}  \right) d\varepsilon.
\end{equation}

According to Eq. (\ref{eq21}) $\phi_i(x,y,\varepsilon)$ is
\textit{invariant} if and only if $d\phi_i(x,y,\varepsilon) = 0$,
i.e. if

\begin{equation}
\label{eq24}
\begin{aligned}
& \dfrac{\partial \phi_i}{\partial x} + \dfrac{\partial
\phi_i}{\partial y}\dfrac{\partial F}{\partial x} = 0 \quad
\Leftrightarrow \quad \dfrac{\partial F}{\partial x} = - \dfrac{
\dfrac{\partial \phi_i}{\partial x} }{ \dfrac{\partial
\phi_i}{\partial y} },\\ & \dfrac{\partial \phi_i}{\partial
\varepsilon} + \dfrac{\partial \phi_i}{\partial y}\dfrac{\partial
F}{\partial \varepsilon} = 0 \quad \Leftrightarrow \quad
\dfrac{\partial F}{\partial \varepsilon} = - \dfrac{ \dfrac{\partial
\phi_i}{\partial \varepsilon} }{ \dfrac{\partial \phi_i}{\partial y}
}.
\end{aligned}
\end{equation}

By replacing $y=F\left( x, \varepsilon \right)$ by its expression in
both parts of Eq. (\ref{eq25}) and by setting

\[
a_{10} =-\left. {\dfrac{\dfrac{\partial \phi_i}{\partial
x}}{\dfrac{\partial \phi_i}{\partial y}}} \right|_{y = F\left(
{x,\varepsilon }\right)} \quad \mbox{and} \qquad a_{01} =-\left.
{\dfrac{\dfrac{\partial \phi_i}{\partial
\varepsilon}}{\dfrac{\partial \phi_i}{\partial y}}} \right|_{y =
F\left( {x,\varepsilon }\right)},
\]

we have that

\begin{equation}
\label{eq25} \begin{array}{*{20}c} \dfrac{\partial F\left(
{x,\varepsilon } \right)}{\partial x} = F'_0 (x) + \mbox{
}\varepsilon F'_1(x) +O\left( {\varepsilon^2} \right)= a_{10},
\vspace{4pt}  \\ \dfrac{\partial F\left( {x,\varepsilon }
\right)}{\partial \varepsilon } = F_1(x) + 2\varepsilon F_2(\e)
+O\left(
{\varepsilon^2} \right)= a_{01}.  \\
\end{array}
\end{equation}

By using a \textit{recurrence reasoning} it may be easily stated
that the functions $F_i(x)$ of the \textit{slow invariant manifold}
associated with a two-dimensional singularly perturbed dynamical
system are given by the expressions (\ref{eq21}).
\end{proof}

\subsection{Van der Pol's ``canards''}

We consider again the Van der Pol system (\ref{eq12}). All functions
$F_i(x)$ of the perturbation expansion may be deduced from the
\textit{slow manifold} equation defined by (\ref{eq16}). But, since
the determination of $F_3(x)$, i.e. the computation $\mu_2$ requires
a \textit{third order tensor of curvature} we consider $\phi_3 (
{\vec {X}, \varepsilon} )$ the second time derivative of $\phi_1 (
{\vec {X}, \varepsilon} )$ which corresponds to the third order
approximation in $\varepsilon$.

\newpage

Thus, we find at

\vspace{0.1in}

\textbf{Order $\varepsilon^0$}

\[
F'_0 (x) =\mathop {\lim}\limits_{\varepsilon \to 0} \left[ a_{10}
\right] = -1 +x^2,
\]

from which one deduces that

\[
F_0(x) = \dfrac{x^3}{3} - x + C_0,
\]

where the constant $C_0$ may be chosen in such a way that the
\textit{critical manifold} can be found again ($C_0 = 0$).

\vspace{0.1in}

\textbf{Order $\varepsilon^1$}

\begin{equation}
\label{eq26} F_1 (x) = \mathop {\lim}\limits_{\varepsilon \to 0}
\left[ a_{01} \right] = \dfrac{\mu_0 - x}{x^2 - 1}.
\end{equation}

Thus, one find again exactly the same functions $F_1(x)$ as those
given by \textit{Geometric Singular Perturbation Method}
(\ref{eq14}) and of course the same value of the bifurcation
parameter $\mu_0 = 1$.

\vspace{0.1in}

\textbf{Order $\varepsilon^2$}

\begin{equation}
\label{eq27} F_2(x) = \dfrac{\mu_1 + \dfrac{(x - \mu_0)(x^2 + 1 -
2\mu_0x)}{(x^2 - 1)^3}}{x^2 - 1}.
\end{equation}

Taking into account that $\mu_0 = 1$ we find again exactly the same
functions $F_2(x)$ as those given by \textit{Geometric Singular
Perturbation Method} (\ref{eq15}) and of course the same value of
the bifurcation parameter $\mu_1 = -1/8$.

\smallskip

\textbf{Order $\varepsilon^3$}

A simple and direct computation leads to $\mu_2 = -\dfrac{3}{32}$.
Thus, the bifurcation parameter value leading to canard solutions
reads

\[
\mu = 1 -\dfrac{1}{8} \varepsilon  -\dfrac{3}{32} \varepsilon^2  +
O\left( {\varepsilon ^3} \right).
\]

Then, for $\varepsilon = 0.01$ one finds again the value obtained by
Beno\^{i}t \textit{et al.} \cite[p.99]{BenCallDinDin} and Diener
\cite{Diener}

\[
\mu = 0.99874...
\]

A program made with Mathematica and available at:
{http://ginoux.univ-tln.fr} enables to compute all order of
approximations in $\varepsilon$ of any two-dimensional singularly
perturbed systems.

\section{Conclusion}

Thus, the bifurcation parameter value leading to a \textit{canard explosion} in dimension two obtained by the so-called \textit{Geometric Singular Perturbation
Method} has been found again with the \textit{Flow Curvature Method}. This result could be also extended to three-dimensional singularly perturbed dynamical systems such as the 3D-autocatalator in which canard phenomenon occurs.

\section*{Acknowledgments}

Authors would like to thank Professor Eric Beno\^{i}t and the referees for their fruitful advices.

The second author is supported by the grants MIICIN/FEDER MTM
2008--03437, AGAUR 2009SGR410, and ICREA Academia.

\end{document}